\documentclass[10pt]{article}

\usepackage[margin=1.2in]{geometry}
\usepackage[colorlinks,unicode]{hyperref}
\usepackage[all]{xy}
\usepackage[nottoc,numbib]{tocbibind}
\usepackage{amsmath,amsthm,amsfonts,longtable,verbatim,multicol,amssymb,wasysym,setspace,graphicx,titlesec,ulem,mathtools}
\usepackage[inline]{enumitem}

\newcommand{\bb}{\medbreak}
\newcommand{\nt}{\noindent}
\newcommand{\R}{\mathbb{R}}

\newcommand{\Z}{\mathbb{Z}}
\newcommand{\N}{\mathbb{N}}
\newcommand{\Cc}{\mathbb{C}}
\newcommand{\rt}{\xrightarrow{}}
\newcommand{\xrt}{\xrightarrow}

\newcommand{\cd}{\cdot}
\newcommand{\id}{\text{id}}

\newcommand{\GL}{\text{GL}}

\newcommand{\End}{\text{End}}

\newcommand{\rank}{\text{rank}}
\newcommand{\ux}{\underline{x}}
\newcommand{\uy}{\underline{y}}

\newcommand{\Alg}{\mathop{\mathrm{Alg}}}

\newcommand{\Mod}{\text{-}\mathrm{Mod}}
\newcommand{\cl}{\text{Cl}}
\newcommand{\bbw}{\overline{\overline{W}}}

\newcommand{\bw}{\overline{W}}
\newcommand{\bq}{\overline{q}}
\newcommand{\CR}{\mathcal{R}}

\newcommand{\define}[1]{\textbf{#1}}

\newtheorem{lemma}{Lemma}[section]
\newtheorem{theorem}[lemma]{Theorem}
\newtheorem{cor}[lemma]{Corollary}

\newtheorem{proposition}[lemma]{Proposition}

\theoremstyle{definition}
\newtheorem{definition}[lemma]{Definition}
\newtheorem{example}[lemma]{Example}
\newtheorem{remark}[lemma]{Remark}

\titleformat{\section}
  {\centering\normalfont\fontsize{10}{15}\bfseries}{\thesection}{1em}{}
 
\titleformat{\subsection}
  {\normalfont\fontsize{11}{15}\bfseries}{\thesubsection}{1em}{}
  
\titleformat{\subsection}[runin]{\normalfont\bfseries}{\thesubsection.}{1em}{}{}
\titlespacing{\subsection}{0pt}{5pt}{10pt}

\title{Twists of rational Cherednik algebras}
\author{Y. Bazlov, A. Berenstein, E. Jones-Healey and A. McGaw}
\date{}

\begin{document}

\maketitle

\begin{abstract} 
We show that braided Cherednik algebras introduced by the first two authors are cocycle twists of rational Cherednik algebras of the imprimitive complex reflection groups $G(m,p,n)$, when $m$ is even. This gives a new construction of mystic reflection groups which have Artin-Schelter regular rings of quantum polynomial invariants. As an application of this result, we show that a braided Cherednik algebra has a finite-dimensional representation if and only if its rational counterpart has one.
\end{abstract}

\tableofcontents


\section{Introduction}

Cocycle twists of associative (and Lie) algebras have their origin in physics literature. Situations when a parametrised family of isomorphic groups of symmetry have a non-isomorphic group as a limit were formalised as ``contractions'' in Inonu and Wigner \cite{Inonu1953}: e.g., the Galilei group of classical mechanics is a limit of relativistic Lorentz groups. Moody and Patera \cite{Moody_1991} show that for graded Lie algebras, contractions are determined by 2-cocycles on the grading group. See also Vafa and Witten \cite{VAFA1995189} where twists by cocycles of a finite abelian group are held as examples of mirror symmetry. There are various applications of cocycle twists within mathematics, for example in non-commutative geometry (see Davies \cite{2015arXiv150206101D}) and colour Lie algebras (see Chen, Silvestrov and Oystaeyen \cite{chen2006representations}).  Cocycle twists also find a generalisation in the language of Hopf algebras, leading to a twist originally due to Drinfeld~\cite{drinfeld_quantum_groups}.  The  Drinfeld twist has been well-studied, see Majid \cite{majid_1995}, and has also found applications in representation theory, see Giaquinto and Zhang \cite{giaquinto1998bialgebra} and Jordan \cite{2008arXiv0805.2766J}. To ascribe physical meaning to twists of an algebra, representations of the algebra also need to be twisted; but this proves to be more difficult, and there is no general approach to this so far.\bb
 
\nt We see the main result of this work, that rational and braided Cherednik algebras are related via a twist, as a stepping stone towards a better understanding of the representation theory of braided Cherednik algebras \cite{Bazlov2008NoncommutativeDO}, for which very little is currently known. In this work we present one result in this direction, showing that finite-dimensional representations of one algebra exist if and only if they do for the twisted partner. We also give an example using one-dimensional representations of the Cherednik algebra, to show that twisting can non-trivially permute the characters of the underlying reflection group.\bb

\nt The contents of this paper are laid out as follows. Sections \ref{reflection_groups_sec}, \ref{cherednik_algs_sec}, \ref{drinfeld_twist_sec} introduce the main definitions and objects of the paper, with the only new result being Theorem~\ref{mystic_presentation}, where we give a presentation of the mystic reflection groups $\mu(G(m,p,n))$. These groups arose independently in the work of Bazlov and Berenstein \cite{Bazlov2008NoncommutativeDO}, as the groups over which the braided Cherednik algebras are defined, and of Kirkman, Kuzmanovich and Zhang \cite{Kirkman2008ShephardToddChevalleyTF}
as a class of groups with Artin-Schelter regular rings of quantum polynomial invariants. Mystic reflection groups were comprehensively studied in \cite{mystic_reflections}. The main result of the paper is found in Section \ref{main_result_sec}, where the braided Cherednik algebra over $\mu(G(m,p,n))$
is shown to be the twist of the rational Cherednik algebra
over the imprimitive complex reflection group $G(m,p,n)$ 
by a cocycle (in fact, a quasitriangular structure) on a finite abelian group. 
A key step in the proof is to verify that the twist preserves the braid relations between the mystic reflection generators of $\mu(G(m,p,n))$ --- this fact turns out to be related to the Clifford Braiding Theorem of Kauffman and Lomonaco~\cite{kauffman2016braiding}.
Finally in Section \ref{twists_of_reps_sec} we use this twisting construction to obtain examples of non-trivial, and finite-dimensional, representations of a braided Cherednik algebra out of representations of a rational Cherednik algebra.


\section{Reflection groups}\label{reflection_groups_sec}

\subsection{Complex reflection groups.} Let $V$ be an $n$-dimensional $\Cc$-vector space, with dual space $V^*$. If $y\in V^*,\ x\in V$, we denote the evaluation of $y$ on $x$ by $\langle y,x\rangle$. If $G$ is a finite subgroup of $\GL(V)$, then for $g\in G,\ x\in V$ we denote the action of $g$ on $x$ by $g(x)$. Via the contragredient representation we have an action of $G$ on $V^*$: if $g\in G,v\in V^*$, then $\langle g(y),x\rangle=\langle y,g^{-1}(x)\rangle\ \forall x\in V$.\bb

\nt A complex reflection on $V$ is an element $s\in \GL(V)$ that has finite order and satisfies $\rank(s-\id)=1$. Equivalently, the characteristic polynomial for $s$ is $(t-1)^{n-1}(t-\lambda)$ for some root of unity $\lambda\neq 1$. Note that in this case $s$ acts on $V^*$ also as a complex reflection with characteristic polynomial $(t-1)^{n-1}(t-\lambda^{-1})$. A complex reflection group on $V$ is a finite subgroup of $\GL(V)$ generated by complex reflections on $V$.\bb 

\nt We fix a basis $\{x_1,\dots,x_n\}$ for $V$, and dual basis $\{y_1,\dots,y_n\}$ of $V^*$. This allows us to identify $\GL(V)$ with the group $\GL_n(\Cc)$ of $n\times n$-invertible matrices. The groups of permutation matrices and diagonal matrices are given respectively as
$$
\mathbb{S}_n=\{w\in \GL(V)\mid \forall i\ w(x_i)\in \{x_1,\dots,x_n\}\},
\qquad 
\mathbb{T}_n=\{t\in \GL(V)\mid \forall i\ t(x_i)\in \Cc x_i\}.
$$
We see that
$$
\mathbb{S}_n\cap \mathbb{T}_n=\{\id\},\qquad
\mathbb{S}_n \text{ normalises } \mathbb{T}_n \text{ within } \GL_n(\Cc), \qquad 
\text{so that } \mathbb{S}_n\ltimes \mathbb{T}_n\subseteq \GL_n(\Cc).
$$
In particular, $\mathbb S_n$ acts on $\mathbb T_n$ 
by conjugation inside $\GL_n(\Cc)$; we will write this action as $w(t)$ for $w\in \mathbb S_n$, $t\in \mathbb T_n$.\bb

\nt For parameters $m,p\in \N$ with $p\mid m$, let $C_{\frac{m}{p}}\subseteq C_m\subset \Cc^\times$ be the finite multiplicative subgroups of $\Cc^\times$ of $\frac{m}{p}$-th, respectively $m$-th, roots of unity. Besides the $34$ exceptional cases, every irreducible complex reflection group belongs to the following family of imprimitive subgroups of $\mathbb{S}_n\ltimes \mathbb{T}_n$,
\begin{equation*}
G(m,p,n)= \mathbb{S}_n\ltimes T(m,p,n), 
\end{equation*}
where
$$
T(m,p,n) = \{ t\in \mathbb{T}_n \mid  t^m=\id, \, \det(t)\in C_{\frac{m}{p}}\}
$$
is the group of diagonal matrices  with diagonal entries in $C_m$ whose product is in $C_{\frac{m}{p}}$.
The complex reflections of $G(m,p,n)$ are given by
$$
S=\{s_{ij}^{(\epsilon)}\mid i,j\in [n], i \ne j,\ \epsilon\in C_m\}\cup \{t_i^{(\zeta)}\mid i\in [n], \zeta\in C_{\frac{m}{p}}\backslash \{1\}\}
$$
where $[n]:=\{1,2,\dots,n\}$, and for general $\epsilon\in \Cc^\times$,
\begin{equation*}
s_{ij}^{(\epsilon)}(x_k):=
\begin{cases}
x_k, & k\neq i,j,\\
\epsilon^{-1} x_j, & k=i,\\
\epsilon x_i, & k=j.
\end{cases}\qquad t_i^{(\epsilon)}(x_k)=\epsilon^{\delta_{ik}}x_k.
\end{equation*}

\nt Note that the groups $G(1,1,n)$, $G(2,1,n)$, $G(2,2,n)$ correspond to the Coxeter groups of type $A_{n-1}$, $B_n$, $D_n$ respectively, whilst $G(p,p,2)$ corresponds to the dihedral group $I_2(p)$.\bb
\nt It will be relevant to the construction of the rational Cherednik algebras in which conjugacy class each complex reflection lies in. The reflections $s_{ij}^{(\epsilon)}$ with $\epsilon\in C_m$ are involutions and form a single conjugacy class in $G(m,p,n)$, unless $n=2$ and $p$ is even. 
Additionally for each $\zeta\in C_{\frac{m}{p}}\backslash \{1\}$, the $\{t_i^{(\zeta)}\mid i\in [n]\}$ forms a separate conjugacy class.\bb 

\nt Recall also that complex reflection groups are characterised in terms of their polynomial invariants. If $G$ is a finite subgroup of $\GL(V)$, the action of $G$ on $V$ extends naturally to algebra automorphisms of the symmetric algebra $S(V)$. The invariant set $S(V)^G$ forms a subalgebra of $S(V)$.

\begin{theorem}[Chevalley-Shephard-Todd]\label{chev_shep_todd} For $V$ and $G$ as above, the invariant ring $S(V)^G$ is a polynomial algebra if and only if $G$ is a complex reflection group.
\end{theorem}

\subsection{Mystic reflection groups.}

Each of the complex reflection groups $G(m,p,n)$, for $m$ even, has a so called ``mystic partner", which is another subgroup of $\mathbb{S}_n\ltimes \mathbb{T}_n$, defined as follows:
\begin{equation*}
\mu(G(m,p,n)):=\{wt\in \mathbb{S}_n\ltimes T(m,1,n)\mid \det(wt)\in C_{\frac{m}{p}}\}.
\end{equation*}
These are examples of mystic reflection groups, which are defined by generalising the Chevalley-Shephard-Todd characterisation of complex reflection groups (Theorem \ref{chev_shep_todd}), see \cite{Kirkman2008ShephardToddChevalleyTF}. The same class of groups was obtained independently in \cite{Bazlov2008NoncommutativeDO}, see also \cite{mystic_reflections}.

\begin{definition}\label{mystic_reflection_group} For a matrix $q\in \text{Mat}_n(\Cc)$ with $q_{ij}q_{ji}=1,\ q_{ii}=1$, let $S_q(V)$ be the algebra generated by $V$ subject to relations $x_ix_j=q_{ij}x_j x_i$ for $i,j\in [n]$. A finite group $G$ is a \define{mystic reflection group} if it has a faithful action by degree-preserving automorphisms on $S_q(V)$ such that the invariant subalgebra $S_q(V)^G$ is isomorphic to $S_{q'}(V)$ for some $q'$.
\end{definition}


\nt The groups $G(m,p,n)$ are related to their mystic partners in the following ways:
\begin{equation}\label{mystic_complex_reln}\mu(G(m,p,n)) \begin{cases}
=G(m,p,n) & \text{if }\frac{m}{p} \text{ is even},\\
\cong G(m,p,n) & \text{if } \frac{m}{p}\text{ and } n \text{ are odd,}\\
\not\cong G(m,p,n) & \text{if } \frac{m}{p}\text{ is odd and } n \text{ is even.}
\end{cases}\end{equation}

\nt Even though the groups $\mu(G(m,p,n))$ and $G(m,p,n)$ in some cases coincide as a subgroup of $\mathbb{S}_n\ltimes \mathbb{T}_n$, 
the generating set for $\mu(G(m,p,n))$ relevant for what follows 
is the set 
\begin{equation}\label{mystic_reflections}
\underline{S}:=\{\sigma_{ij}^{(\epsilon)} \mid i,j\in [n],i\neq j,\ \epsilon\in C_m\}\cup \{t_i^{(\zeta)}\mid i\in [n], \zeta\in C_{\frac{m}{p}}\backslash \{1\}\}
\end{equation}
where
\begin{equation}\label{sigma_map}\sigma_{ij}^{(\epsilon)}(x_k)=\begin{cases}
x_k & k\neq i,j,\\
\epsilon^{-1} x_j & k=i,\\
-\epsilon x_i & k=j.
\end{cases}
\end{equation}
We call the elements of $\underline{S}$ \define{mystic reflections}. Notice $\sigma_{ij}^{(\epsilon)}\in \mathbb{S}_n\ltimes \mathbb{T}_n$ are of order $4$, with characteristic polynomial $(x-1)^{n-2}(x^2+1)$. Similarly to above, when $n\geq 3$, the $\sigma_{ij}^{(\epsilon)}$ form a single conjugacy class, whilst for each $\zeta\in C_{\frac{m}{p}}\backslash \{1\}$ the $t_i^{(\zeta)}$ again form separate conjugacy classes. 


\subsection{A presentation of mystic reflection groups.}\label{presentations_sec}

It turns out that for fixed $n$, each mystic reflection group $\mu(G(m,p,n))$ contains the Tits group of type $A_{n-1}$, introduced in \cite[Section 4.6]{TITS196696}. The Tits group is realised as $\mu(G(2,2,n))$, the mystic partner of the Coxeter group of type~$D_n$; it is the group of even elements in the Coxeter group $G(2,1,n)$ of type $B_n$.

\begin{theorem}[a presentation of $\mu(G(m,p,n))$]\label{mystic_presentation}
For all even $m$ and all divisors $p$ of $m$,
the abstract group generated by symbols
$\sigma_1,\dots,\sigma_{n-1}$ and the abelian group $T(m,p,n)$,
subject to the relations
\begin{itemize}
	\item[(i)] $\sigma_i^2 = t_i^{(-1)} t_{i+1}^{(-1)}$,
	\item[(ii)] the braid relations $\sigma_i\sigma_j=\sigma_j\sigma_i$,
	$i-j>1$,
	$\sigma_i\sigma_{i+1}\sigma_i =\sigma_{i+1}\sigma_i\sigma_{i+1}$,
	$1\le i\le n-2$, and
	\item[(iii)] $\sigma_i t \sigma_i^{-1} = s_{i,i+1}(t)$
	for all $t\in T(m,p,n)$
\end{itemize}
is isomorphic to $\mu(G(m,p,n))$ via the map 
$\sigma_i\mapsto \sigma_{i,i+1}^{(1)}$ and 
the identity map on $T(m,p,n)$.
\end{theorem}
\begin{proof}
Let $W$ be the quotient of the free product 
of a free group on $\{\sigma_i\mid  i\in [n-1]\}$ with $T(m,p,n)$ by relations (i)--(iii). If $m=p=2$, then by \cite[Lemma 2.1]{ADAMS2017142}, $W$ is the Tits group $\mathcal T\subset \mathrm{SL}_n(\Cc)$ of type $A_{n-1}$, which surjects 
onto $\mathbb S_n$ with kernel $T(2,2,n)$. Therefore, for $m$ even, the subgroup of $W$ generated by 
$\sigma_1,\dots,\sigma_{n-1}$ and $T(2,2,n)$
is some quotient $\overline{\mathcal T}$ of $\mathcal T$.
Moreover, by rearranging generators using relation (iii)  we can write $W$ as $\overline{\mathcal T}\cdot \overline Q$ where $Q$ is a 
transversal of $T(2,2,n)$ in $T(m,p,n)$. 
Hence
$$
|W| \le |\mathcal T|\cdot |Q| 
= |\mathbb S_n||T(2,2,n)| \cdot |T(m,p,n)| / | T(2,2,n)|
= |\mathbb S_n||T(m,p,n)|.
$$
Now observe that relations (i)--(iii) hold in  $\mu(G(m,p,n))$:
one checks them using the factorisation 
$\sigma_{i,i+1}^{(1)}=s_{i,i+1}t_{i+1}^{(-1)}$ in $\mathbb S_n \ltimes \mathbb T_n$, the 
Coxeter relations for 
the $s_{i,i+1}$ and the semidirect product relations.
Hence the map $W \to \mu(G(m, p, n))$, given in the Theorem, is well-defined. 
We show that this map is surjective, i.e., $\sigma_{i,i+1}^{(1)}$, $i\in [n-1]$ and elements of $T(m,p,n)$ generate the group $\mu(G(m,p,n))$.
The composite homomorphism $\mu(G(m,p,n)) \hookrightarrow \mathbb S_n \ltimes \mathbb T_n \xrightarrow{\mathrm{proj}_1} \mathbb S_n$ carries the $\sigma_{i,i+1}^{(1)}$ to the 
generators $s_{i,i+1}$ of $\mathbb S_n$ so is surjective with kernel 
$T(m,p,n)$. Hence the $\sigma_{i,i+1}^{(1)}$ generate $\mu(G(m,p,n))$ modulo $T(m,p,n)$, as required.
We thus have
$$
|W| \le |\mathbb S_n||T(m,p,n)| = |\mu(G(m,p,n))|,
$$
so the surjective homomorphism $W\to \mu(G(m,p,n))$ is a bijection.
\end{proof}

\section{Rational and braided Cherednik algebras}\label{cherednik_algs_sec}

\subsection{Rational Cherednik algebras.}

For complex reflection group $G\subseteq \GL(V)$, let $S$ denote the complex reflections in $G$. Let $c:S\rt \Cc, s\mapsto c_s$ be such that $c_{gsg^{-1}}=c_s\ \forall g\in G,s\in S$. Using $c$, we define a bilinear map
$$\kappa_c:V^*\times V\rt \Cc G,\ (y,x)\mapsto \langle y,x\rangle 1 +\sum_{s\in S}c_s \langle y,(\id-s)x\rangle s.
$$
The following algebras were introduced by Etingof and Ginzburg in \cite{etingof2002symplectic}:
\begin{definition} \label{def:rca}
	The \define{rational Cherednik algebra} $H_c(G)$ is generated by $V,\Cc G,V^*$, subject to the relations: $\forall x,x'\in V,\ y,y'\in V^*,\ g\in G$,
\begin{itemize}
    \item $xx'-x'x=yy'-y'y=0$
    \item $gx=g(x)g,\ yg=g\cd g^{-1}(y)$
    \item $yx-xy=\kappa_c(y,x)$
\end{itemize}
\end{definition} 

\nt Note that when $c=0$, the algebra $H_0(G)=\mathcal{A}(V)\# \Cc G$, the smash product of the Weyl algebra with the group algebra $\Cc G$. Whilst $H_0(G)$ has no finite-dimensional modules, $H_c(G)$ can have finite-dimensional modules for special values of $c$.\bb

\begin{theorem}[The PBW theorem for rational Cherednik algebras, \protect{\cite[Theorem 1.3]{etingof2002symplectic}}]
\label{thm:pbw_rational}	
Let $x_1,\dots,x_n$ be basis of $V$, and $y_1,\dots,y_n$ a dual basis for $V^*$. As a $\Cc$-vector space, $H_c(G)$ has basis
$$
\{x_1^{k_1}\dots x_n^{k_n}g y_1^{l_1}\dots y_n^{l_n}
\mid g\in G,\ k_1,\dots,k_n,l_1,\dots,l_n\in \Z_{\geq 0}\}.
$$
In other words, as vector spaces, we have $H_c(G)\cong \Cc[x_1,\dots,x_n]\otimes \Cc G\otimes \Cc[y_1,\dots,y_n]$.
\end{theorem}

\nt In the following we restrict to the case $G=G(m,p,n)$ in which $m$ is even, and either $n\geq 3$, or $p$ is odd and $n=2$. This means the $s_{ij}^{(\epsilon)}$ form a single conjugacy class, and we denote the value of $c:S\rt \Cc$ on this class as $c_1$. The groups with odd $m$ are excluded because they have no mystic partner. If $n=2$ and $p$ is even, the algebra defined below is not the most general case of the rational Cherednik algebra because it only has a single parameter for the set of complex reflections of the form $s_{ij}^{(\epsilon)}$, although this set is split into two conjugacy classes in $G(m,p,2)$.

\begin{definition}\label{def:RCA} The \define{rational Cherednik algebra} $H_c(G(m,p,n))$ is the algebra generated by $x_1,\dots,x_n\in V$, $g\in G(m,p,n)$, $y_1,\dots,y_n\in V^*$, subject to relations:
\begin{align*}
x_ix_j-x_jx_i=0\qquad
y_iy_j-y_jy_i=0\qquad
gx_i=g(x_i)g\qquad 
g y_i =g(y_i) g\\
y_ix_j-x_jy_i =c_1\sum_{\epsilon\in C_m} \epsilon s_{ij}^{(\epsilon)}\qquad 
y_ix_i-x_iy_i =1-c_1\sum_{j\neq i}\sum_{\epsilon\in C_m}s_{ij}^{(\epsilon)}-\sum_{\zeta\in C_{\frac{m}{p}}\backslash \{1\}}c_\zeta t_i^{(\zeta)}
\end{align*}
where $c_\zeta$ the value of $c$ on conjugacy class of $t_i^{(\zeta)}$, for each $\zeta \in C_{\frac{m}{p}}$.
\end{definition}

\subsection{Braided Cherednik algebras.}

Consider the mystic reflection group $\mu(G(m,p,n))$, with mystic reflections $\underline{S}$ as in \eqref{mystic_reflections}. We require $m$ to be even in order for the mystic reflection group $\mu(G(m,p,n))$ to be defined; we also assume that $n\ge 3$ or $p$ is odd and $n=2$ as above, so that the mystic reflections $\sigma_{ij}^{(\epsilon)}$
form a single conjugacy class. Similarly to above, we consider a function $c':\underline{S}\rt \Cc$ that is invariant under conjugation in $\mu(G(m,p,n))$, and let ${\ux}_1,\dots,{\ux}_n$ be a basis of $V$, with $\uy_1,\dots,\uy_n$ a dual basis for $V^*$.
The following algebras were introduced in \cite{Bazlov2008NoncommutativeDO}:

\begin{definition} 
	\label{def:negative_braided}
	The \define{negative braided Cherednik algebra} $\underline{H}_{c'}(\mu(G(m,p,n)))$ is the algebra generated by ${\ux}_1,\dots,{\ux}_n\in V$, $g\in \mu(G(m,p,n))$, $\uy_1,\dots,\uy_n\in V^*$, subject to relations:
\begin{align*}{\ux}_i{\ux}_j+{\ux}_j{\ux}_i=0\qquad
  \uy_i\uy_j+\uy_i\uy_j=0\qquad
  g{\ux}_i=g({\ux}_i)g\qquad
  g\uy_i=g(\uy_i)g\\
  \uy_i{\ux}_j+{\ux}_j\uy_i=c'_1 \sum_{\epsilon\in C_m}\epsilon \sigma_{ij}^{(\epsilon)}\qquad
  \uy_i {\ux}_i-{\ux}_i \uy_i =1+c'_1 \sum_{j\neq i}\sum_{\epsilon\in C_m}\sigma_{ij}^{(\epsilon)}+\sum_{\zeta \in C_{\frac{m}{p}}\backslash \{1\}}c'_\zeta t_i^{(\zeta)}
\end{align*}
\end{definition}

\nt The key property of $\underline{H}_{c'}(\mu(G(m,p,n))$ is the PBW-type theorem \cite[Theorem 0.2]{Bazlov2008NoncommutativeDO}. Below we obtain a new proof of this result, see Remark~\ref{rem:new_proof}.



\section{The cocycle twist}\label{drinfeld_twist_sec}

Although we will only use cocycles on a finite abelian group, we will work in a more general Hopf algebra setting as it provides the useful language of duality. 
Our notation generally follows \cite{majid_1995}.
If $H$ is a Hopf algebra over $\Cc$, $\triangle\colon H \to H\otimes H$ will denote the coproduct of~$H$ and $\epsilon\colon H \to \Cc$, the counit. An example is the group algebra $H=\Cc T$ of a group $T$, with $\triangle(t)=t\otimes t$ and $\epsilon(t)=1$, extended from $T$ to $\Cc T$ by linearity. The action of $h\in H$ on $a\in A$ where $A$ is an $H$-module will be written as $h\rhd a$. Recall that in an $H$-module algebra $A$, the product map $m\colon A\otimes A \to A$ is a morphism in the category $H\Mod$ of $H$-modules, and $h\rhd 1_A=\epsilon(h)1_A$ for all $h\in H$.

\subsection{Quasitriangular structures and $2$-cocycles.}
We begin with two well-known definitions, see \cite[Definition 2.1.1 and Example 2.3.1]{majid_1995}. 

\begin{definition} A \define{$2$-cocycle} on a Hopf algebra $H$ is an invertible $\chi\in H\otimes H$ such that $(\chi\otimes 1)\cdot (\triangle \otimes \id)(\chi)=(1\otimes \chi)\cdot (\id \otimes \triangle)(\chi)$ and $(\epsilon \otimes \id)(\chi)=1=(\id\otimes \epsilon)(\chi)$.
\end{definition}

\begin{definition} A \define{quasitriangular (QT-) structure} on a Hopf algebra $H$ is an invertible element $\CR\in H\otimes H$ satisfying:
\begin{itemize}
\setlength{\itemindent}{.1in}
	\item[(QT1)] $\CR\cdot (\triangle x)=(\triangle^{\text{op}}x)\cdot \CR\ \forall x\in H$,
	\item[(QT2)] $(\triangle\otimes \id) \CR =\CR_{13}\cdot \CR_{23}$, 
	$(\id\otimes \triangle) \CR =\CR_{13}\cdot \CR_{12}$,
\end{itemize}
where $\CR_{12}:=\CR\otimes 1,\ \CR_{23}:=1\otimes \CR$, and $\CR_{13}$ similarly has $1$ inserted in the middle leg. 
\end{definition}

\nt QT-structures satisfy the quantum Yang-Baxter equation \cite[Lemma 2.1.4]{majid_1995}
\begin{equation}\label{eq:QYBE} 
\CR_{12}\cdot\CR_{13}\cdot\CR_{23} = \CR_{23}\cdot\CR_{13}\cdot \CR_{12}.
\end{equation}
(QT2) and \eqref{eq:QYBE} imply that a QT-structure is a Hopf algebra $2$-cocycle \cite[Example 2.3.1]{majid_1995}.

\subsection{Twists.}

It is natural to complement Majid's description of a twisted $H$-module algebra \cite[Section 2.3]{majid_1995} by the observation that twisting by $\chi$ is functorial:

\begin{proposition}\label{prop:twist_is_functorial}
A $2$-cocycle $\chi\in H\otimes H$ for a Hopf algebra $H$ gives rise to the functor 
$$(\ )_\chi: \Alg(H\Mod) \to \Alg(H_\chi\Mod)$$
which takes an object $(A,m)$ to $(A,m_\chi = m(\chi^{-1}\rhd - ))$, and an arrow $(A,m)\xrightarrow{\phi}(B,m')$ to $(A,m_\chi)\xrightarrow{\phi_\chi}(B,m'_\chi)$,
where $\phi_\chi=\phi$ as $H$-module morphisms. The twisted Hopf algebra $H_\chi$ is defined as having the same algebra structure, and counit, as $H$, but with coproduct $\triangle_\chi(h)=\chi\cdot \triangle(h)\cdot \chi^{-1}$ where $\triangle$ is the coproduct on $H$.
\end{proposition}

\begin{proof}
That $(A,m_\chi)$ is in $\Alg(H_\chi\Mod)$ is \cite[Proposition 2.3.8]{majid_1995}, so we only need to check functoriality. Since $A\otimes A\xrightarrow{\phi\otimes\phi}B\otimes B$ 
is an $H\otimes H$-module morphism, it commutes with the action of $\chi^{-1}\in H\otimes H$, so 
$\phi\circ m_\chi = m\circ(\phi\otimes\phi) (\chi^{-1} \rhd)=m(\chi^{-1} \rhd) (\phi\otimes\phi)$ shows that $\phi_\chi$ is an algebra morphism if $\phi$ is. Also, $\phi_\chi$ is an $H_\chi$-module morphism because the actions of $H$ and of $H_\chi$ are the same. Therefore, $\phi_\chi$ is indeed an arrow in $\Alg(H_\chi\Mod)$.
\end{proof}
\begin{remark}
The functor given in Proposition~\ref{prop:twist_is_functorial} is 
essentially the restriction of the monoidal equivalence 
of categories $H\Mod \to H_\chi\Mod$, see  \cite[Remark 5.14.3]{etingof2016tensor}, to algebras.
\end{remark}


\subsection{The cocycle $\mathcal F$.}\label{subsec:cocycle_F}

We will now define the cocycle $\mathcal F$ which will be used for twisting in the rest of the paper. Let $T$ be the abelian group
$$
T=\langle \gamma_1,\dots,\gamma_n \mid \gamma_i^2=1, 
\,\gamma_i\gamma_j = \gamma_j\gamma_i,\,i,j\in [n]\rangle,
$$
isomorphic to $T(2,1,n)\cong (C_2)^n$. Define 
\begin{equation*} 
f_{ij} = \frac{1}{2}\left(1\otimes 1 + \gamma_i\otimes 1 + 1\otimes \gamma_{j} - \gamma_i\otimes \gamma_j\right),
\qquad 
\mathcal F = \prod_{1\leq j < i \leq n} f_{ij}\,.
\end{equation*}
Let $a,b$ be elements of some associative algebra.
It is easy to check that 
\begin{equation}\label{eq:involutions}
	\text{if $a,b$ are commuting involutions, then }\frac12(1+a+b-ab)\text{ is an involution,}
\end{equation}
which implies that the $f_{ij}$ and $\mathcal F$ are pairwise commuting involutions in $\Cc T \otimes \Cc T$. That they are cocycles
follows from
\begin{lemma}
$\mathcal F$ and $f_{ij}$ for all $i,j$  are quasitriangular structures
on $\Cc T$.
\end{lemma}
\begin{proof} (QT1) is vacuous as $\Cc T$ is a commutative and cocommutative Hopf algebra.
Rewriting $f_{ij}$ in the form $2^{-1}\sum_{a,b=0}^1 (-1)^{ab}\gamma_i^a\otimes \gamma_j^b$, 
one checks (QT2) for $f_{ij}$ in the same way as in the case $n=2$ of \cite[Example 2.1.6]{majid_1995}.
Since (QT2) is multiplicative in $\mathcal R$, (QT2) also holds for $\mathcal F$.
\end{proof}
%

\subsection{$\mathcal F$-twisted product of $T$-eigenvectors.}

Since $T$ is a finite abelian group, a $\Cc T$-module is the same as a comodule 
for the dual Hopf algebra $\Cc \hat T$, 
where the dual group $\hat T$ of $T$ is
$$
\hat{T}=\{\alpha:T\rt \Cc^\times \mid \alpha \text{ is a group homomorphism}\} = 
\langle \alpha_1,\dots,\alpha_n \rangle,
\qquad \alpha_i(\gamma_j) = (-1)^{\delta_{ij}}.
$$
A coaction by the group algebra of~$\hat T$ manifests itself as a $\hat T$-grading, so for a $\Cc T$-module algebra $A$ we have $A=\bigoplus_{\alpha\in \hat{T}}A_\alpha$ 
where $A_\alpha:=\{a\in A\mid \forall \gamma\in T,\ \gamma\rhd a=\alpha(\gamma)a\}$ is the $\alpha$-eigenspace of $T$. 
Denote by $\star$ 
the twisted product on $A$ induced by the cocycle $\mathcal{F}
\in \Cc T\otimes \Cc T$ as in Proposition~\ref{prop:twist_is_functorial}. Then 
\begin{equation}\label{eq:prod_eigenvectors}
a\star b=(\alpha\otimes \beta)(\mathcal{F}^{-1})a\cdot b
\qquad\text{for $a\in A_\alpha$, $b\in A_\beta$.} 
\end{equation}
Here $\alpha,\beta\in \hat{T}$ induce an algebra homomorphism $\alpha\otimes \beta\colon \Cc T^{\otimes 2}\rt \Cc$. Observe that 
\begin{equation}\label{eq:sign_rule}
(\alpha_i \otimes \alpha_j)(f_{ab}) = 1\text{ if $(i,j)\ne(a,b)$},
\qquad (\alpha_i \otimes \alpha_j)(f_{ij}) = -1.
\end{equation}
Every one of the $2^n$ elements of $\hat T$ is of the form 
$$
\alpha_I := \prod_{i\in I} \alpha_i \qquad\text{where }
I\subseteq [n],
$$
so \eqref{eq:prod_eigenvectors} and \eqref{eq:sign_rule} imply the following
%
%
%
\begin{lemma}\label{lem:star_eigenvectors}
Let $I,J$ be subsets of $[n]$.
If $a\in A_{\alpha_I}$, $b\in A_{\alpha_J}$, then 
$a\star b = (-1)^{d(I,J)} ab$ where 
$d(I,J)=|\{(i,j):\, i\in I,\, j\in J,\, i>j\}|$. \qed
\end{lemma}

\begin{remark}\label{rem:clifford}
In fact, Lemma~\ref{lem:star_eigenvectors} is a particular case of the construction in \cite[Lemma 3.6]{bazlov2013cocycle} where a twist $A_{\mathcal F}$ of a $G$-graded algebra $A$ by a $2$-cocycle $\mathcal F$ on the group $G$ is realised as the image of the coaction $A \to A\otimes \Cc G$ viewed as a 
subalgebra of $A\otimes (\Cc G)_{\mathcal F}$. In the case $G=\hat T$ and the cocycle $\mathcal F$ given above, the twisted group algebra $(\Cc \hat T)_{\mathcal F}$ is isomorphic to the complex Clifford algebra $\mathit{Cl}_n$ of a space with an orthonormal basis $\alpha_1,\dots,\alpha_n$, \cite[Example 1.7]{bazlov2013cocycle}. The calculations done in the next section can be interpreted as embedding the negative braided Cherednik algebra $\underline H_{\underline c}(\mu(G(m,p,n)))$ in $H_c(G(m,p,n))\otimes \mathit{Cl}_n$, although we do not explicitly write the Clifford algebra generators. 
\end{remark}


\section{The Main Result}\label{main_result_sec}

\subsection{Statement of the main theorem.}
To state Main Theorem~\ref{main_result} below, we 
need to define the action of the abelian group $T$ 
on the rational Cherednik algebra $H_c(G(m,p,n))$. 
This will allow us to twist the associative product 
in $H_c(G(m,p,n))$.
\begin{proposition}\label{prop:action}
Let $m$ be even and $n\geq 2$. The rational Cherednik algebra $H_c(G(m,p,n))$ is a $\Cc T$-module algebra with respect to the action $\rhd$ given by:
$$\gamma_i\rhd g=t_i^{(-1)}gt_i^{(-1)}\qquad \gamma_i\rhd x_j=t_i^{(-1)}(x_j),\qquad \gamma_i\rhd y_j=t_i^{(-1)}(y_j)$$
for all $i,j\in [n],\ g\in G(m,p,n)$.
\begin{proof} 

The	PBW theorem~\ref{thm:pbw_rational} and the defining relations in Definition~\ref{def:RCA} imply that
$H_c(G(m,p,n))$ embeds as a subalgebra in 
$\widetilde H = H_{\widetilde c}(G(m,1,n))$, where 
$\widetilde c$ is defined by $\widetilde c_1=c_1$, 
$\widetilde c_\zeta = c_\zeta$ for $\zeta \in C_{\frac mp}
\setminus \{1\}$ and $\widetilde c_\zeta = 0$
whenever $\zeta\in C_m\setminus C_{\frac mp}$.
This subalgebra is generated by $x_i$, $y_i$ for $i\in [n]$ and by the subgroup $G(m,p,n)$ of $\widetilde G = G(m,1,n)$.
Observe that $\widetilde H$ is a $\Cc \widetilde G$-module algebra 
where $a\in \widetilde G$ acts by conjugation, 
carrying $u\in \widetilde H$ to $aua^{-1}\in \widetilde H$.
%
%
%
%
Since $G(m,p,n)$ is a normal subgroup of $\widetilde G$, this action of $\widetilde G$ preserves 
the subalgebra $H_c(G(m,p,n))$ of $\widetilde H$. 
Now the embedding of $T$ in $\widetilde G$ via $\gamma_i\mapsto t_i^{(-1)}$ 
defines the $T$-action on $H_c(G(m,p,n))$ given in the Proposition.
\end{proof}
\end{proposition}

\nt 
We now twist the rational Cherednik algebra by the cocycle $\mathcal{F}\in \Cc T \otimes \Cc T$ from Section~\ref{drinfeld_twist_sec} and  denote the result $(H_c(G(m,p,n)),\star)$.

\begin{theorem}\label{main_result} 
There exists an isomorphism
$$\phi \colon \underline H_{\underline{c}}(\mu(G(m,p,n))) \to (H_c(G(m,p,n)),\star)$$
of associative algebras, where $\underline{c}:\underline{S}\rt \Cc$ is defined by $\underline{c}_1:=-c_1,\ \underline{c}_\zeta:=-c_\zeta\ \forall \zeta\in C_{\frac{m}{p}}\backslash \{1\}$, and 
$$
\phi(\sigma_i)=\bar s_i,\qquad
\phi(t) = t,\qquad
\phi({\ux}_j) = x_j, \qquad 
\phi(\uy_j)=y_j, 
$$
for all $i\in [n-1]$, $j\in [n]$, $\sigma_i:=\sigma_{i,i+1}^{(1)}$, $\bar s_i:=s_{i,i+1}^{(-1)}$, $t\in T(m,p,n)$.
\end{theorem}

\subsection{Outline of the proof of the Theorem.}

We fix the triple $(m,p,n)$ and denote $G=G(m,p,n)$, $H_c=H_c(G)$, $\underline{H}_{\underline c}=\underline{H}_{\underline c}(\mu(G))$. 
The Theorem defines the algebra homomorphism $\phi$ on generators of $\underline{H}_{\underline c}$, and so $\phi$, if exists, is unique. We first extend $\phi$ from the generators $\sigma_i$ and $t$ of $\mu(G)$ to 
a homomorphism from $\mu(G)$ to a subgroup of the twisted group algebra $(\Cc G,\star)$.
Then, by checking that the defining relations of $\underline{H}_{\underline c}$ from Definition~\ref{def:negative_braided} are satisfied by the elements $\phi({\ux}_i)$, $\phi(\uy_i)$ and $\phi(g)$, $g\in \mu(G)$ of the algebra $(H_c,\star)$, 
we show that $\phi$ extends from the generators to the whole algebra $\underline{H}_{\underline c}$. We then use the PBW theorem for $H_c$ to argue that $\phi$ is bijective.

\subsection{$\star$-multiplication by simple generators.}

We need several lemmas where we express 
the new, $\mathcal F$-twisted associative product $\star$ of certain elements of $H_c$ in terms of the usual product (written as $\cdot$ or omitted).
\begin{lemma}\label{lem:starisdotfort}
	For all $t\in T(m,p,n)$, $u\in H_c$, 
	$t\star u=tu$ and $u\star t=ut$. 
\end{lemma}
\begin{proof} 
As $t$ is invariant under the action of $T$,
i.e., is a $T$-eigenvector with eigencharacter $1=\alpha_{\emptyset}$, by Lemma~\ref{lem:star_eigenvectors}, $t\star u=tu$ and $u\star t=ut$ for all $T$-eigenvectors $u$, and so by linearity for all $u$ in $H_c$.
\end{proof}

\nt We denote
$$
s_i := s_{i,i+1} \in \mathbb S_n, 
\quad r_{ij} := t_i^{(-1)} t_j^{(-1)}\in T(m,p,n),
\quad p_i =\frac12(1+\gamma_i), 
\quad q_i=\frac12(1-\gamma_i)
\in \Cc T
$$
and let 
$$
s_i^+ := p_i \rhd s_i = \frac12 (s_i + \bar s_i),
\quad 
s_i^- := q_i \rhd s_i = \frac12(s_i-\bar s_i) \in \Cc G.
$$
Observe that $\gamma_a \rhd s_i = s_i$ if $a\notin \{i,i+1\}$, so $s_i^+$ is $T$-invariant, and $s_i^-\in (H_c)_{\alpha_i \alpha_{i+1}}$. 

\begin{lemma}\label{lem:s_au}
For all $i\in [n-1]$, $u\in H_c$,
$$s_i \star u = s_i\cdot (p_i \rhd u) + \bar s_i \cdot (q_i \rhd u)
\qquad\text{and}\qquad 
u\star s_i = (p_{i+1}\rhd u) s_i+(q_{i+1}\rhd u)\bar s_i.
$$
\end{lemma}
\begin{proof}
Since the expressions are linear in $u$, we may assume that $u$ is a $T$-eigenvector with eigencharacter $\alpha_J$, $J\subseteq [n]$. Apply Lemma~\ref{lem:star_eigenvectors} to $s_i^-\in (H_c)_{\alpha_i \alpha_{i+1}}$ and $u$. In the string $i, i+1, J$, the pair $i,i+1$ forms zero or two inversions with every element of $J$ except possibly $i$, and forms exactly one inversion with $i$ if $i\in J$, so $s_i^-\star u= s_i^- u$ if $i\notin J$, and 
$s_i^- \star u=-s_i^- u$ if $i\in J$. That is, $s_i^- \star u = s_i^- \cdot (\gamma_i \rhd u)$.\bb

\nt Also by Lemma~\ref{lem:star_eigenvectors},
$s_i^+\star u = s_i^+ u$. The formula for $s_i\star u=(s_i^++s_i^-)\star u$ follows. The proof for $u\star s_i$ is similar.
\end{proof}

\begin{lemma}\label{lem:braid_relations}
The simple transpositions $s_i$, $i\in [n-1]$, obey the relations $s_i\star s_i=r_{i,i+1}$ and the braid relations with respect to the $\star$-product on the group algebra of $G(m,p,n)$.
\end{lemma}
\begin{proof}
By Lemma~\ref{lem:s_au}, $s_i \star s_i = s_i s_i^+ + \bar s_i s_i^- = \frac12 (1+r_{i,i+1})+\frac12 r_{i,i+1}(1-r_{i,i+1})=r_{i,i+1}$.
\bb

\nt If $|i-j|>1$, then $s_j$ is $\gamma_i$-invariant, hence $q_i\rhd s_j=0$. 
Then by Lemma~\ref{lem:s_au}, $s_i \star s_j =s_i s_j$. This is symmetric in $i,j$, so $s_i$ and $s_j$ $\star$-commute.\bb

\nt To check the braid relation for $s_i$ and $s_{i+1}$, we can assume $i=1$. We calculate the product 
\begin{equation}\label{eq:twoways}
	s_2 \star s_{1,3} \star s_1 = (s_2 \star s_2s_1s_2) \star s_1 = s_2 \star (s_1s_2s_1 \star s_1)
\end{equation}
in two ways. First, since the transposition $s_{1,3}=s_2s_1s_2$ is $\gamma_2$-invariant, by Lemma~\ref{lem:s_au} 
$s_2 \star s_2s_1s_2 = s_2s_2s_1s_2 = s_1s_2$. This is the same as $s_1\star s_2$ because $s_1$ is $\gamma_3$-invariant, so~\eqref{eq:twoways} equals 
$s_1\star s_2\star s_1$. \bb

\nt On the other hand, by $\gamma_2$-invariance of
$s_{1,3}=s_1s_2s_1$ and the second part of Lemma~\ref{lem:s_au}, $s_1s_2s_1\star s_1
=s_1s_2s_1s_1=s_1s_2 = s_1\star s_2$, so~\eqref{eq:twoways} equals $s_2\star s_1 \star s_2$.
The $\star$-braid relations are proved.
\end{proof}
\begin{remark}
	If $(H_c,\star)$ is embedded in $H_c\otimes \mathit{Cl}_n$ as in Remark \ref{rem:clifford},
	the simple transposition $s_i$ becomes 
	$s_i^+\otimes 1 + s_i^-\otimes \alpha_i\alpha_{i+1}$.
The calculation to prove Lemma~\ref{lem:braid_relations} is then equivalent to verifying part of the Clifford Braiding Theorem of Kauffman and Lomonaco \cite{kauffman2016braiding} which states that the $e_i:=\frac{1+\alpha_i\alpha_{i+1}}{\sqrt2}$, $i\in [n-1]$, obey the braid relations
in $\mathit{Cl}_n$. The Clifford Braiding Theorem goes further to assert 
the circular braid relations involving $e_n:=1+\alpha_n\alpha_1$, but these do not arise from Lemma~\ref{lem:braid_relations}.
\end{remark}

\subsection{The extension of $\phi$ from the generators to the group algebra $\Cc \mu(G)$.}\label{group_relations_proof_sec}

To prove that the assignment $\phi(\sigma_i)=\bar s_i$, $\phi(t)=t$
extends to a homomorphism $\phi\colon \Cc \mu(G) \to (\Cc G,\star)$ of algebras, we check relations (i)--(iii) from the presentation of $\mu(G)$ given in Theorem~\ref{mystic_presentation}.\bb

\nt (i) $\sigma_i \sigma_i=r_{i,i+1}$.
We check that the relation $\phi(\sigma_i)\star \phi(\sigma_i)=\phi(r_{i,i+1})$
holds in $(H_c,\star)$. The left-hand side is
$\bar s_i\star \bar s_i=(s_i r_{i,i+1})\star (
s_i r_{i,i+1})$, which by Lemma~\ref{lem:starisdotfort}
is $s_i\star s_i$. By Lemma~\ref{lem:braid_relations}, this is $r_{i,i+1}$, the same as 
$\phi(r_{i,i+1})$.\bb

\nt (ii) $\sigma_i\sigma_j=\sigma_j\sigma_i$ for $i-j>1$. 
We need to check that $\phi(\sigma_i)=\bar s_i$ and $\phi(\sigma_j)=\bar s_j$ $\star$-commute.
This follows from Lemma~\ref{lem:starisdotfort} and Lemma~\ref{lem:braid_relations}.\bb

\nt $\sigma_i\sigma_{i+1}\sigma_i=\sigma_{i+1}\sigma_i \sigma_{i+1}$. 
By Lemma~\ref{lem:braid_relations}, 
$s_i\star s_{i+1}\star s_i=s_{i+1}\star s_i\star s_{i+1}$, and by Proposition~\ref{prop:twist_is_functorial}, $(H_c,\star)$ is a $T$-module algebra. Acting by $\gamma_{i+1}$ on both sides gives the required relation $\bar s_i\star\bar  s_{i+1}\star\bar  s_i=\bar s_{i+1}\star\bar  s_i\star\bar  s_{i+1}$.\bb


\nt (iii) $\sigma_i t\sigma_i^{-1}=s_{i}(t)$. We need to check that 
$\bar s_i\star t\star s_i=s_i(t)$. 
The left-hand side rewrites by Lemma~\ref{lem:starisdotfort} as $(\bar s_i\star s_i)s_i(t)
=r_{i,i+1}(\bar s_i\star \bar s_i)s_i(t)$, which simplifies to $s_i(t)$ by~(i).\bb

\nt We can now describe the map $\phi$ on the whole of $\Cc \mu(G)$ using a special basis of $\Cc \mu(G)$: 
\begin{proposition}\label{prop:theta_w}
There exist involutions $\theta_w\in \Cc T(2,1,n)$,
indexed by $w\in \mathbb S_n$, such that $\{w\theta_w t: w\in \mathbb S_n, t\in T(m,p,n)\}$ is a basis of $\Cc\mu(G)$. In this basis, 
\begin{equation}\label{eq:phi_formula}
	\phi(w\theta_w t)=wt.
\end{equation} 
\end{proposition}
\begin{proof} First of all, we observe that, for each 
$i\in[n-1]$ and $w\in \mathbb S_n$,
$$
\bar s_i \star w = s_iw \cdot \eta_i(w)
$$
for some involution $\eta_i(w)\in \Cc T(2,2,n)$. Indeed, denote $t_i^{(-1)}\cdot w^{-1}(t_i^{(-1)})$ by $r$ so that $\gamma_i\rhd w=wr$. By Lemma~\ref{lem:starisdotfort} and Lemma~\ref{lem:s_au}, 
$$
\bar s_i\star w=s_ir_{i,i+1}\cdot \frac12(w+wr) + s_i 
\cdot \frac12 (w-wr) = s_iw\cdot \frac12(1-r+r'+rr')
$$
with $r'=w^{-1}(r_{i,i+1})$. Thus,  $\eta_i(w)=\frac12(1+r+r'-rr')$, which is an involution by~\eqref{eq:involutions}.\bb

\nt Factorise $w$ into simple transpositions as $w=s_{i_k}\dots s_{i_2} s_{i_1}$, and let $\sigma_w = \sigma_{i_k}\dots \sigma_{i_2}\sigma_{i_1}$.
Since $\sigma_i = s_i t_{i+1}$ and $\det \sigma_i=1$, in the group 
$\mathbb S_n\ltimes T(2,1,n)$ one has 
$\sigma_w = wt_w$ with $t_w\in T(2,1,n)$ such that $\det(t_w)=\det(w)$. Therefore,
$$
\phi(wt_w) = \bar s_{i_k}\star \dots \star \bar s_{i_2}\star 
\bar s_{i_1} = s_{i_k}\dots s_{i_2} s_{i_1} \cdot 
\eta_{i_2}(s_{i_1})\eta_{i_3}(s_{i_2} s_{i_1})\dots 
\eta_{i_k} (s_{i_{k-1}}\dots s_{i_1}), 
$$
so \eqref{eq:phi_formula} holds with 
$\theta_w=\eta_{i_2}(s_{i_1})\dots 
 \eta_{i_k}(s_{i_{k-1}}\dots s_{i_1} ) t_w$.
Note that $wt_w$ and $\eta_{i_j}(s)$, hence $w\theta_w$, lie in $\Cc G(2,2,n)$ and so $\{w\theta_w t: w\in \mathbb S_n, t\in T(m,p,n)\}$ is a subset of $\Cc \mu(G)$. It is a basis of the space $\Cc \mu(G)$, because this set is carried by the linear map $\phi$ to the basis $\{wt\}$ of the space $\Cc G$ of the same dimension.
\end{proof}

\subsection{Commutation relations between the \texorpdfstring{${\ux}_i$}{x} and between the \texorpdfstring{$\uy_j$}{y}.}
We need to show that
$\phi({\ux}_i)\star\phi({\ux}_j) = - \phi({\ux}_j)\star\phi({\ux}_i)$, 
$\phi(\underline{y}_i)\star\phi(\underline{y}_j) = - \phi(\underline{y}_j)\star\phi(\underline{y}_i)$
whenever $1\le i<j\le n$. 
This is immediate by the following 
\begin{cor}\label{x_ix_j}
	For all $i,j\in [n]$, $i<j$, 
	$$	 
	x_i\star x_i = x_i^2, \quad 
	x_i \star x_j = -x_j \star x_i = x_ix_j, \quad
	x_i \star y_i = x_i y_i, 
	\quad 
	x_i \star y_j = x_iy_j, 
	\quad y_j \star x_i =  -y_j x_i. 
	$$
	The same holds where the letters $x$ and $y$ are swapped.
\end{cor}
\begin{proof}
Since the $x_i$ and $y_i$ are $T$-eigenvectors with eigencharacter $\alpha_i$ this follows by Lemma \ref{lem:star_eigenvectors}.
\end{proof}

\subsection{The main commutator relation between \texorpdfstring{${\ux}_i$}{x} and \texorpdfstring{$\uy_j$}{y}.}
We will now check the  relation obtained 
by applying $\phi$ to both sides of the main commutator relation in $\underline{H}_{\underline{c}}$
for $i\ne j$:
\begin{equation}\label{eq:main_commutator_ij}
\phi(\uy_i)\star \phi({\ux}_j) + \phi({\ux}_j)\star \phi(\uy_i) = \underline{c}_1\displaystyle\sum \limits_{\epsilon\in C_m} \epsilon \phi(\sigma_{ij}^{(\epsilon)}).
\end{equation}
To calculate the right-hand side, we need 
\begin{lemma}\label{phi_sigma_epsilon_result} $\phi(\sigma_{ij}^{(\epsilon)})=s_{ij}^{(-\epsilon)}$ if $i<j$, and 
$\phi(\sigma_{ij}^{(\epsilon)})=s_{ij}^{(\epsilon)}$ if $i>j$.
\begin{proof}
Since $\sigma_{ij}^{(\epsilon)}=
\sigma_{ij}^{(1)}t_i^{(\epsilon^{-1})}t_j^{(\epsilon)}$,  
$s_{ij}^{(\pm \epsilon)}
=s_{ij}^{(\pm 1)}\star t_i^{(\epsilon^{-1})}t_j^{(\epsilon)}$ 
and $\phi (t_i^{(\epsilon^{-1})}t_j^{(\epsilon)})=t_i^{(\epsilon^{-1})}t_j^{(\epsilon)}$, 
it is enough to prove the Lemma for $\epsilon=1$.\bb

\nt The case $i<j$: if $j=i+1$, the statement becomes $\phi(\sigma_i)=\bar s_i$ which is true by definition of~$\phi$. To proceed by induction in $j$, we consider the identity $s_j s_{i,j+1}^{(-1)} = s_{ij}^{(-1)}s_j$. 
Since $s_{i,j+1}^{(-1)}$ is $\gamma_j$-invariant, and 
$s_{ij}^{(-1)}$ is $\gamma_{j+1}$-invariant, 
this rewrites by Lemma~\ref{lem:s_au} as 
$s_j \star s_{i,j+1}^{(-1)} = s_{ij}^{(-1)}\star s_j$.
Using $\bar s_j\star s_j=1$,
$$
 s_{i,j+1}^{(-1)} = \bar s_j\star s_{ij}^{(-1)}\star s_j
 = \phi(\sigma_j \sigma_{ij}^{(1)} \sigma_j^{-1}) = \phi(\sigma_{i,j+1}^{(1)}),
$$
so the inductive step is done, and the case $i<j$ follows.\bb

\nt The case $i > j$: $\phi(\sigma_{ij}^{(\epsilon)}) = \phi(\sigma_{ji}^{(-\epsilon^{-1})})$, which by the first part of the Lemma is  $s_{ji}^{(\epsilon^{-1})}=s_{ij}^{(\epsilon)}$.
\end{proof}
\end{lemma}
\nt 
Now, by Corollary~\ref{x_ix_j} and Lemma~\ref{phi_sigma_epsilon_result} relation \eqref{eq:main_commutator_ij} for $i<j$ rewrites as
$$
y_i x_j - x_j y_i = \underline c_1 
\sum\limits_{\epsilon\in C_m} \epsilon s_{ij}^{(-\epsilon)}.
$$
Substituting $\underline c_1=-c_1$ and recalling that $-1\in C_m$, and so $-\epsilon \in C_m$ iff $\epsilon\in C_m$, transforms this equation into the relation between $y_i$ and
$x_j$ in Definition~\ref{def:RCA}. 
If $i>j$, \eqref{eq:main_commutator_ij} becomes 
$-(y_ix_j-x_jy_i)=-c_1\sum_{\epsilon\in C_m} \epsilon s_{ij}^{(\epsilon)}$, 
which is again true by Definition~\ref{def:RCA}.
Thus, \eqref{eq:main_commutator_ij} is proved.

\subsection{The main commutator relation between \texorpdfstring{${\ux}_i$}{x} and \texorpdfstring{$\uy_i$}{y}.}

We need to show that
\begin{equation}\label{main_commutator}
\phi(\uy_i) \star\phi({\ux}_i) - \phi({\ux}_i)\star \phi(\uy_i) = 1 + \underline{c}_1 \displaystyle\sum\limits_{j\not=i}\displaystyle\sum\limits_{\epsilon\in C_m} \phi(\sigma_{ij}^{(\epsilon)}) + \displaystyle\sum\limits_{\zeta\in C_{\frac{m}{p}}\setminus\{1\}} \underline{c}_\zeta \phi(t_{i}^{(\zeta)})
\end{equation}
where the left-hand side is $y_ix_i-x_iy_i$ by Corollary~\ref{x_ix_j}.
Apply Lemma~\ref{phi_sigma_epsilon_result}, note
that $\sum_{\epsilon\in C_m}s_{ij}^{(-\epsilon)}$ is the same as  
$\sum_{\epsilon\in C_m}s_{ij}^{(\epsilon)}$ because $-1\in C_m$,
and substitute
$\phi(t_i^{(\zeta)})=t_i^{(\zeta)}$, $\underline{c}_\zeta =-c_\zeta$
to rewrite the right-hand side of~\eqref{main_commutator} 
as  $1-c_1\sum_{j\neq i}\sum_{\epsilon\in C_m}s_{ij}^{(\epsilon)}-\sum_{\zeta\in C_{\frac{m}{p}}\setminus \{1\}}c_\zeta t_i^{(\zeta)}$. This shows that  \eqref{main_commutator} is true by the relation for $y_ix_i-x_iy_i$ in Definition~\ref{def:RCA}.

\subsection{The semidirect product relations.}
We need to prove 
\begin{equation}\label{eq:semidirect}
	\phi(g)\star x_i = g(x_i) \star \phi(g), \qquad
	\phi(g)\star y_i = g(y_i) \star \phi(g)
\end{equation}
for all $g\in \mu(G)$ and all $i\in [n]$. We will prove this only for $x_i$, as the proof for $y_i$ will be similar. Moreover, observe that if \eqref{eq:semidirect} holds for $g=g_1$ and 
for $g=g_2$, then it holds for $g=g_1g_2\in \mu(G)$:
$$
\phi(g)\star x_i = \phi(g_1)\star\phi(g_2) \star x_i
=\phi(g_1)\star g_2(x_i)\star \phi(g_2)
= g_1(g_2(x_i)) \star \phi(g_1) 
\star \phi(g_2) = 
g(x_i) \star \phi(g).
$$
By Theorem~\ref{mystic_presentation}, every element of
$\mu(G)$ is a product of some generators $\sigma_j$, $j\in [n-1]$, and of some $t\in T(m,p,n)$.  
If $g=t\in T(m,p,n)$, one can omit $\phi$ and $\star$ in 
\eqref{eq:semidirect} by Lemma~\ref{lem:starisdotfort}, 
and then \eqref{eq:semidirect} clearly holds by the semidirect product relations in $H_c$. Hence it is enough to prove \eqref{eq:semidirect} when $g=\sigma_j$:
\begin{equation}\label{eq:semidirect_simple}
\phi(\sigma_j)\star x_i  =\sigma_j(x_i)\star\phi(\sigma_j), 
\quad\text{ equivalently } \bar s_j\star x_i = \sigma_j(x_i)\star \bar s_j.
\end{equation}
If $i\notin \{j,j+1\}$, then $x_i = \sigma_j(x_i)$ is both 
$\gamma_j$ and $\gamma_{j+1}$-invariant, 
so by Lemma~\ref{lem:s_au}, \eqref{eq:semidirect_simple} rewrites as $\bar s_j x_i = x_i \bar s_j$.\bb

\nt If $i=j+1$, $x_{j+1}$ is $\gamma_j$-invariant and  
$\sigma_j(x_{j+1})=-x_j$ is $\gamma_{j+1}$-invariant, 
so by Lemma~\ref{lem:s_au}, \eqref{eq:semidirect_simple}
is $\bar s_j x_{j+1} = -x_j \bar s_j$.\bb

\nt If $i=j$, $x_j$ is a $T$-eigenvector with eigencharacter $\alpha_j$ and $\sigma_j(x_j)=x_{j+1}$, with $\alpha_{j+1}$, so by Lemma~\ref{lem:s_au} \eqref{eq:semidirect_simple}
is $s_j x_j = x_{j+1} s_j$. In all three cases, 
\eqref{eq:semidirect_simple} is true by the semidirect product relations in~$H_c$.

%
%
%

%

\subsection{Bijectivity of $\phi$.}

Hence all the relations are satisfied and $\phi$ is a well-defined algebra homomorphism. We are left to prove that $\phi$ is bijective.
It is enough to construct a spanning set of $\underline{H}_{\underline c}$
which is carried by $\phi$ to a basis of $H_c$.\bb
%
%


\nt Let $w\in \mathbb S_n$. Consider the coset $w T(m,p,n)$ of $T(m,p,n)$ inside $G$ and let $\langle w  T(m,p,n) \rangle$ denote the span of this coset,  a subspace of $\Cc G$. Observe that $\langle w T(m,p,n)\rangle$ is a $T$-submodule of $\Cc G$, because, if $i\in [n]$, $t\in T(m,p,n)$, $\gamma_i \rhd (wt)=wrt$ where $r=t_i^{(-1)}\cdot w^{-1}(t_i^{(-1)})\in T(2,2,n)\subset T(m,p,n)$. Therefore, $\langle w T(m,p,n)\rangle$ has $T$-eigenbasis $wb_1(w),\dots,wb_N(w)$ where $N=|T(m,p,n)|$ and $b_1(w),\dots,b_N(w)\in \Cc T(m,p,n)$. It follows that $\{wb_m(w): w\in\mathbb S_n,m\in[N]\}$ is a basis of the group algebra $\Cc G=\oplus_{w\in\mathbb S_n}\langle w T(m,p,n)\rangle$. The PBW-type tensor product factorisation of $H_c$, see Theorem \ref{thm:pbw_rational}, implies that
\begin{equation}\label{main_prf_rca_best_basis}
\mathcal B = \{ x_{1}^{k_1}\ldots x_{n}^{k_n} w b_m(w) y_{1}^{l_1}\ldots y_{n}^{l_n}: w \in \mathbb S_n, m\in [N], k_{i},l_{i}\geq 0\}
\end{equation}
is a basis of $H_c$.\bb

\nt We replace the basis $\{w\theta_w t\}$ of $\Cc \mu(G)$, given by Proposition \ref{prop:theta_w}, by the following alternative basis: $\{w\theta_w b_m(w): w\in \mathbb S_n, m\in [N]\}$. 
It is a basis of $\Cc \mu(G)$ because by Proposition~\ref{prop:theta_w}, 
it is carried by $\phi$ to the basis  $\{wb_m(w)\}$ of $\Cc G$, 
and $\dim \Cc \mu(G)=\dim \Cc G$. 
It then follows from the defining relations in $\underline H_{\underline c}$ that the set
\begin{equation}\label{main_prf_nbca_best_basis}
\underline{\mathcal B}=\{{\ux}_{1}^{k_1}\ldots {\ux}_{n}^{k_n}
w \theta_w b_m(w) {\uy}_{1}^{l_1}\ldots {\uy}_{n}^{l_n}: w \in \mathbb S_n, m\in [N], k_{i},l_{i}\geq 0\}
\end{equation}
spans $\underline H_{\underline c}$.\bb

\nt It is immediate from Corollary \ref{x_ix_j} that $\phi({\ux}_{1}^{k_1}\ldots {\ux}_{n}^{k_n}) = x_{1}^{k_1}\ldots x_{n}^{k_n}$ and $\phi({\uy}_{1}^{l_1}\ldots {\uy}_{n}^{l_n}) = y_{1}^{k_1}\ldots y_{n}^{l_n} $.
We can now view how a general basis element of \eqref{main_prf_nbca_best_basis} is mapped under $\phi$,
\begin{align*}
\phi({\ux}_{1}^{k_1}\ldots {\ux}_{n}^{k_n}
w \theta_w b_m(w) {\uy}_{1}^{l_1}\ldots {\uy}_{n}^{l_n})
& =x_{1}^{k_1} \ldots x_{n}^{k_n} \star 
w b_m(w) \star y_{1}^{l_1}\ldots y_{n}^{l_n}
\\ & = \pm x_{1}^{k_1} \ldots x_{n}^{k_n}
w b_m(w) y_{1}^{l_1}\ldots y_{n}^{l_n}
\end{align*}
where the second equality follows by Lemma \ref{lem:star_eigenvectors}, 
since $x_{1}^{k_1} \ldots x_{n}^{k_n}$, $w b_m(w)$ and $ y_{1}^{l_1}\ldots y_{n}^{l_n}$ are $T$-eigenvectors. On noting that a spanning set of $\underline H_{\underline c}$ has been mapped (up to a scalar multiple of $\pm 1$) to the basis of $H_c$ given in \eqref{main_prf_rca_best_basis}, we
conclude that $\underline {\mathcal B}$ is a basis of $\underline H_{\underline c}$ and that $\phi$ is bijective, as required.

\begin{remark}\label{rem:new_proof}
The above argument showing that  $\underline {\mathcal B}$ is a basis of $\underline H_{\underline c}$ is a new proof of the PBW-type theorem for  negative braided Cherednik algebras, obtained earlier in \cite{Bazlov2008NoncommutativeDO}. 
\end{remark}

\section{Twists of representations}\label{twists_of_reps_sec}

From Theorem~\ref{main_result} we know rational Cherednik algebras can be twisted, and the result is isomorphic to a negative braided Cherednik algebra. Next we show that representations of rational Cherednik algebras can also be twisted, generating a representation of the corresponding negative braided Cherednik algebra. A systematic approach to twists of representations, going beyond the examples considered in this section, will be explored in the upcoming paper \cite{followonpaper}.

\subsection{Twisting and finite-dimensional representations.}\label{twist_procedure_sec}
For the purposes of this section, we assume $\frac mp$ to be even. Recall that by Proposition \ref{prop:action} $H_c(G(m,p,n))$ is a $\Cc T$-module algebra under the conjugation action of $T=(C_2)^n$. When $\frac{m}{p}$ is even, $T$ is embedded as the subgroup $T(2,1,n)$ in $G(m,p,n)$, so 
a representation
$$\rho:H_c(G(m,p,n))\rt \End(V)$$
of $H_c(G(m,p,n))$ induces a $T$-action on $\End(V)$ via $t \RHD f=\rho(t)f\rho(t)^{-1}$ for all $t\in T$, $f\in \End(V)$. With this, $\rho$ becomes a $\Cc T$-module algebra homomorphism.
Denote by $\End(V)_\mathcal{F}$ the twist of the $\Cc T$-module algebra $\End(V)$ by the cocycle $\mathcal{F}$ defined in Section~\ref{subsec:cocycle_F}.\bb
 
\nt Recall that the underlying vector space is unchanged by twisting, therefore $\rho$
can be viewed as a linear map  $\rho:H_c(G(m,p,n))_\mathcal{F}\rt \End(V)_\mathcal{F}$. 

\begin{proposition}\label{twisted_hom}
The linear map $\rho:H_c(G(m,p,n))_\mathcal{F}\rt \End(V)_\mathcal{F}$ is an algebra homomorphism.
\begin{proof} Let $m$ denote the product on $H_c(G(m,p,n))$ and $m'$ be the product on $\End(V)$, so that $\rho \circ m=m'\circ(\rho\otimes \rho)$ because $\rho$ is a homomorphism between the untwisted algebras. The twisted product maps are 
$m_{\mathcal F} = m\circ (\mathcal F^{-1}\rhd)$ 
and $m'_{\mathcal F} = m'\circ (\mathcal F^{-1}\RHD)$, 
and since $\rho$ is a $\Cc T$-module algebra morphism, 
so that $(\rho\otimes \rho)(\mathcal F^{-1}\rhd)=(\mathcal F^{-1}\RHD)(\rho\otimes \rho)$, 
we conclude that $\rho\circ m_{\mathcal F} = 
m'_{\mathcal F}\circ (\rho\otimes\rho)$, as required.
\end{proof}

\end{proposition}

\nt We use this to deduce the following:

\begin{theorem}
For $\frac{m}{p}$ even, if $H_c(G(m,p,n))$ has finite-dimensional representations, then so does the negative braided Cherednik algebra $\underline{H}_{\underline{c}}(\mu(G(m,p,n)))$.
\end{theorem}
\begin{proof}
If $\rho:H_c(G(m,p,n))\rt \End(V)$ is a finite-dimensional representation, the algebra $\End(V)_\mathcal{F}$ is finite-dimensional, so it has finite-dimensional modules on which $\underline{H}_{\underline{c}}(\mu(G(m,p,n)))\cong H_c(G(m,p,n))_\mathcal{F}$ acts via the algebra homomorphism 
$\rho:H_c(G(m,p,n))_\mathcal{F}\rt \End(V)_\mathcal{F}$.
%
\end{proof}

\subsection{Finite-dimensional representations: a general construction.} 

Let $W$ be an irreducible complex reflection group with reflection representation $V$, and $H_c(W)$ be a rational Cherednik algebra over $W$. 
We recall a general approach which produces finite-dimensional representations 
of $H_c(W)$. Start with a simple $\Cc W$-module $\tau$ 
and extend $\tau$ to an $S(V)\# \Cc W$-module where $V$ acts by zero.
The standard $H_c(W)$-module $M_c(\tau)$ is defined as 
$$
M_c(\tau) = H_c(W)\otimes_{S(V)\# \Cc W}\tau.
$$
The underlying vector space of $M_c(\tau)$ is $S(V^*)\otimes \tau$, hence these are infinite-dimensional representations of $H_c(W)$. The standard module $M_c(\Cc)$ given by $\tau=\Cc$, the trivial $\Cc W$-module, is the famous Dunkl (or polynomial) representation of $H_c(W)$. Every $M_c(\tau)$ has a unique simple quotient, denoted $L_c(\tau)$. For some $\tau$ and some values of $c$, $L_c(\tau)$ are finite-dimensional.

\nt If $x_i$ is a basis of $V$ and $y_i$ the dual basis of $V^*$, one has the following important element of $H_c(W)$:
\begin{equation}\label{h_element}
h=\sum_i x_iy_i+\frac{n}{2}-\sum_{s\in S}\frac{2c_s}{1-\lambda_s}s
\end{equation}
where $n=\dim(V)$, $S$ is the set of complex reflections in $W$, and $\lambda_s$ is the non-trivial eigenvalue of $s$ in the dual reflection representation. 
By \cite[Section 2.1]{2003math3194C}, $h$ satisfies the 
commutator relations
\begin{equation}\label{commutators}
[h,x]=x\ \forall x\in V,\qquad [h,y]=-y\ \forall y\in V^*.
\end{equation}

\subsection{Twisting one-dimensional representations of \texorpdfstring{$H_c(G(2,1,n))$}{HcG21n}.}

We restrict the discussion above to the group $G = G(2,1,n)\cong \mathbb S_n \ltimes (C_2)^n$, the Coxeter group of type $B_n$, and consider 
the modules $L_c(\tau)$ over $H_c=H_c(G)$ which are one-dimensional. Such modules correspond to the four 
linear characters $\mathrm{triv},\kappa, \det$ and $\kappa \det$ of~$G$; each
character is determined by its values on $s=s_i$ and $t=t_i^{(-1)}$
as follows:
\begin{equation}\label{eq:characters}
\mathrm{triv}\colon (s,t) \mapsto (1,1), \quad
\kappa\colon (s,t)\mapsto (1,-1), \quad
\det\colon (s,t)\mapsto (-1,-1), \quad
\kappa \det\colon (s,t) \mapsto (-1,1).
\end{equation}
%

\nt To find the parameters $c$ where $\dim L_c(\tau)=1$, we note that commutators must act on a one-dimensional module by zero, so \eqref{commutators} implies that the generators $x_i$ and $y_i$, $i\in [n]$, act by $0$. Most relations in Definition~\ref{def:RCA} are satisfied by $x_i=y_i=0$ automatically: the relation $y_i x_j - x_j y_i = c_1 \sum_{\epsilon\in \{\pm1\}} \epsilon s_{ij}^{(\epsilon)}$ holds because 
$s_{ij}$ and $s_{ij}^{(-1)}$ act on $L_c(\tau)$ by the same scalar. The only constraint on the parameter~$c=(c_1,c_{-1})$ arises from the last relation which reads
\begin{equation}\label{eq:constraint}
0 = 1 - c_1 \cdot (n-1)\cdot 2\tau(s_i) - c_{-1}\tau(t_i^{(-1)}).
\end{equation}

\nt Let $\tau$ be one of the four characters of $G$ given in~\eqref{eq:characters}, and assume that $c$ satisfies~\eqref{eq:constraint} so that $L_c(\tau)$ is $1$-dimensional. We apply the twisting procedure from Section~\ref{twist_procedure_sec} to the $H_c$-module $L_c(\tau)$. The action of the group $T$ on $\End(L_c(\tau))$ is via conjugation, however, $\End(L_c(\tau))\cong\Cc$ is commutative. Hence $T$ acts trivially, and $\End(L_c(\tau))_{\mathcal F} = \End(L_c(\tau))$.\bb 

\nt We obtain the following $1$-dimensional representation of the 
negative braided Cherednik algebra $\underline H_{\underline c} = \underline H_{\underline c}(\mu(G))$,
\begin{equation}\label{eq:composite_action}
\underline H_{\underline c} \xrightarrow{\phi} (H_c)_{\mathcal F} \xrightarrow{\rho_\tau} \End(L_c(\tau))_{\mathcal F} \cong \Cc.    
\end{equation}
Here $\rho_\tau$ is the algebra homomorphism we arrive at from Proposition~\ref{twisted_hom}. Denote the $1$-dimensional $\underline H_{\underline c}$-module 
given by \eqref{eq:composite_action} by $L_c(\tau)_{\mathcal F}$.\bb

\nt Recall from \eqref{mystic_complex_reln} that the group $\mu(G)$ is the same as the group $G$. To characters $\tau$ of $\mu(G)=G$ there correspond one-dimensional representations $\underline{L}_{\underline c}(\tau)$ of the negative braided Cherednik algebra $\underline H_{\underline c}(G)$
where the ${\ux}_i$ and $\uy_i$ act by $0$ and elements of $G$ act via $\tau$. We can now identify the twists $L_c(\tau)_{\mathcal F}$ 
as certain representations $\underline{L}_{\underline c}(\tau')$
of $\underline H_{\underline c}(G)$, as follows:
$$
\tau'(\sigma_i)  = \rho_\tau(\phi(\sigma_i)) = 
\rho_\tau(\bar s_i) = \rho_\tau(\sigma_i t_i) = \tau(\sigma_i)\tau(t_i),
\qquad
\tau'(t_i) = \rho_\tau(\phi(t_i)) = \rho_\tau(t_i) = \tau(t_i).
$$
This means that twisting induces a non-trivial permutation of 
linear characters of the group $G(2,1,n)$, resulting in the following theorem which concludes the paper:
\begin{theorem}[Twists of one-dimensional representations of $H_c(G(2,1,n))$]
The twisting procedure outlined above maps one-dimensional representations of $H_c(G(2,1,n))$ to one-dimensional representations of $\underline H_{\underline c}(G(2,1,n))$ as follows,
$$
L_c(\mathrm{triv})_{\mathcal F} = \underline L_{\underline c}(\mathrm{triv}), \quad
L_c(\kappa)_{\mathcal F} = \underline L_{\underline c}(\det), \quad
L_c(\det)_{\mathcal F} = \underline L_{\underline c}(\kappa), \quad
L_c(\kappa\det)_{\mathcal F} = \underline L_{\underline c}(\kappa\det).
$$
\end{theorem}

\bibliographystyle{plain}
\bibliography{mybib}

\end{document}